\documentclass[12pt, a4paper]{article}
\usepackage[centertags]{amsmath}
\usepackage{amsfonts}
\usepackage{amssymb}
\usepackage{amsthm}
\usepackage{mathrsfs}
\usepackage{tikz}
\usetikzlibrary{shapes,backgrounds}
\usepackage{graphicx}
\usepackage{setspace}

\usetikzlibrary{arrows}
\usetikzlibrary{calc,trees,positioning,arrows,chains,shapes.geometric,%
	decorations.pathreplacing,decorations.pathmorphing,shapes,%
	matrix,shapes.symbols}
\usetikzlibrary{decorations.pathmorphing,shadows} 
\newtheorem{thm}{Theorem}[section]
\newtheorem{cor}[thm]{Corollary}
\newtheorem{lem}[thm]{Lemma}

\numberwithin{equation}{section}
\author {Ya\c{s}ar Polato\~{g}lu, Yasemin Kahramaner and Arzu Yemi\c{s}ci \c{S}en  }
\title {Some Results of The Class of Functions with Bounded Radius Rotation
\footnotetext{\textit{\textnormal{2010} Mathematics Subject Classification:}
30C45 \\ \textit{Key words and phrases:}  Bounded radius rotation, bounded boundary rotation, distortion theorem, growth theorem and coefficient inequality.
}}
\date{ }

\newcommand{\pe}{{\mathcal P}}
\newcommand{\de}{{\mathbb D}}

\begin{document}
\maketitle

\begin{abstract}
 Let $\mathcal{A}$ be the family of functions $f(z)=z+a_2z^2+...$ which are analytic in the open unit disc $\mathbb{D}=\{z: |z|<1 \}$, and denote by $\pe$ of functions $p(z)=z+p_1z+p_2z^2+...$ analytic in $\de$ such that $p(z)$ is in $\pe$ if and only if
 $$p(z)\prec \frac{1+z}{1-z} \Leftrightarrow p(z)=\frac{1+\phi(z)}{1-\phi(z)}, $$
for some Schwarz function $\phi(z)$ and every $z\in\de.$

Let $f(z)$ be an element of $\mathcal{A}$, and satisfies the condition 
$$z\frac{f'(z)}{f(z)}=\left(\frac{k}{4}+\frac{1}{2}\right)p_{1}(z)-\left(\frac{k}{4}-\frac{1}{2}\right)p_{2}(z)$$
where $p_1(z), p_2(z)\in \pe$ and $k\geq 2$, then $f(z)$ is called function with bounded radius rotation. The class of such functions is denoted by $R_k$. This class is generalization of starlike functions.

The main purpose is to give some properties of the class $R_k$.
\end{abstract}

\maketitle

\section{Introduction}
Let $ \Omega $ be the family of functions $\phi(z)$ which are analytic in  $\de$ and satisfy the conditions $\phi(0)=0$, $\lvert \phi(z)\rvert<1$ for all $z \in \de$. If $f_{1}(z)$ and $f_{2}(z)$ are analytic functions in $\de$, then we say that $f_{1}(z)$ is subordinate to $f_{2}(z)$, written as $f_{1}(z)\prec f_{2}(z)$ if there exists a Schwarz function $\phi\in \Omega\ $ such that $f_{1}(z)=f_{2}(\phi(z)), z\in\de$. We also note that if $f_{2}$ univalent in  $\de$ , then  $f_{1}(z)\prec f_{2}(z)$ if and only if $f_{1}(0) =f_{2}(0)$, $f_{1}(\de) \subset f_{2}(\de)$ implies $f_{1}(\de_{r}) \subset f_{2}(\de_{r})$, where $ \de_{r} =\{z: |z|<r, 0<r<1 \}$ (see \cite{Goodman 1984}). Denote by $\pe$ the family of functions  $p(z)=1+p_{1}z+p_{2}z^{2}+p_{3}z^{3}+\cdots$ analytic in  $\de$ such that $p$ is in $\pe$ if and only if
\begin{equation}
	p(z)\prec \frac{1+z}{1-z} \Leftrightarrow p(z)=\frac{1+w(z)}{1-w(z)}, z\in\de
\end{equation} 
Let $f(z)$ be an element of $\mathcal{A}$. Then $f(z)$ is called convex or starlike if it maps $\mathbb{D}$ onto a convex or starlike region, respectively. Corresponding classes are denoted by $\mathcal{C}$ and $S^{\ast}$. It is well known that $\mathcal{C}\subset S^{\ast},$ that both are subclasses of the univalent functions and have the following analytical representations.
\begin{equation}
	f(z) \in \mathcal{C} \Longleftrightarrow  Re\bigg(1+z\frac{f''(z)}{f'(z)} \bigg)>0,\quad z\in\mathbb{D}
\end{equation}
and
\begin{equation}
	f(z) \in S^{\ast}  \Longleftrightarrow  Re\bigg(z\frac{f'(z)}{f(z)}\bigg )>0,\quad z\in\mathbb{D}
\end{equation}
More on these class can be found in [2]. Let $f(z)$ be an element of $\mathcal{A}$. If there is a function $g(z)$ in $\mathcal{C}$ such that 
\begin{equation}
	Re\bigg(\frac{f'(z)}{g'(z)}\bigg )>0,\quad z\in\mathbb{D}
\end{equation}
then $f(z)$ is called close-to-convex function in $\mathbb{D}$ and the class of such functions is denoted by $\mathcal{CC}$.

A function analytic and locally univalent in a given simply connected domain is said to be of bounded boundary rotation if its range has bounded boundary rotation which is defined as the total variation of the direction angle of the tangent to the boundary curve under a complete circuit. Let $V_k$ denote the class of functions $f(z)\in \mathcal{A}$ which maps $\mathbb{D}$ conformally onto an image domain of boundary rotation at most $k\pi$. The class of functions of bounded boundary rotation was introduced by Loewner \cite{Loewner 1971} in 1917 and was developed by Paatero \cite{Paatero 1931, Paatero 1933} who systematically developed their properties and made an exhaustive study of the class $V_k$. Paatero has shown that $f(z)\in V_k $ if and only if    
\begin{equation}
	f'(z)=Exp\left[-\int^{2\pi}_{0}\log\left(1-ze^{-it}\right)d\mu(t)\right],
\end{equation}
where $\mu(t)$ is real-valued function of bounded variation for which
\begin{equation}
	\int^{2\pi}_{0} d\mu(t)=2 \quad\text{and}\quad \int^{2\pi}_{0} |d\mu(t)| \leq k
\end{equation}
for fixed $k\geq 2$ it can also be expressed as
\begin{equation}\label{e1.10}
	\int^{2\pi}_{0}\left|Re\frac{(zf'(z))'}{f'(z)}\right|d\theta\leq 2k\pi, \quad z=re^{i\theta}.
\end{equation}
Clearly, if $k_{1}<k_{2}$ then $V_{k_{1}}\subset V_{k_{2}}$ that is the class $V_k$ obviously expands on $k$ increases. $V_2$ is the class of $\mathcal{C}$ of convex univalent functions.
Paatero showed that $V_4 \subset \mathcal{S}$, where $\mathcal{S}$ is the class of normalized univalent functions. Later Pinchuk proved that  $V_k$ are close-to convex functions in $\de$ if $2\leq k\leq 4$ \cite{Pinchuk 1971}.

Let $R_k$ denote the class of analytic functions $f$ of the form $f(z)=z+a_2z^2+a_3z^3+...$ having the representation 
\begin{equation}
f(z)=zExp\left[-\int^{2\pi}_{0}\log\left(1-ze^{-it}\right)d\mu(t)\right],
\end{equation}
where $\mu(t)$ is given in (1.6). We note that the class $R_k$ was introduced by Pinchuk and Pinchuk showed that Alexander type relation between the classes $V_k$ and $R_k$ exists,
\begin{equation}
f\in V_k \Leftrightarrow zf'(z) \in R_k
\end{equation}
$R_k$ consists of those function $f(z)$ which satisfy 
\begin{equation}\label{e1.10}
\int^{2\pi}_{0}\left|Re(re^{i\theta}\frac{f'(re^{i\theta})}{f(re^{i\theta})})\right|d\theta\leq k\pi, z=re^{i\theta}.
\end{equation}
Geometrically, the condition is that the total variation of angle between radius vector $f(re^{i\theta})$ makes with positive real axis is bounded $k\pi$. Thus, $R_k$ is the class of functions of bounded radius rotation bounded by $k\pi$, therefore  $R_k$ generalizes the starlike functions.

$P_k$  denote the class of functions $p(0)=1$ analytic in $\de$ and having representation 
\begin{equation}\label{e1.10}
p(z)=\int^{2\pi}_{0}\frac{1+ze^{-it}}{1-ze^{-it}}d\mu(t)
\end{equation}
where $\mu(t)$ is given in (1.6). Clearly, $P_2=P$ where $P$ is the class of analytic functions with positive real part. For more details see \cite{Pinchuk 1971}. From (1.11), one can easily find that $p(z)\in P_k$ can also written by
\begin{equation}
 p(z)=\left(\frac{k}{4}+\frac{1}{2}\right)p_{1}(z)-\left(\frac{k}{4}-\frac{1}{2}\right)p_{2}(z),\quad z\in\de
\end{equation}
where $p_{1}(z),p_{2}(z)\in\pe$. Pinchuk \cite{Pinchuk 1971}  has shown that the classes $V_k$ and  $R_k$ can be defined by using the class $P_k$ as gives below 
\begin{equation}
f\in V_k \Leftrightarrow \frac{(zf'(z))'}{f'(z)} \in P_k
\end{equation}
and
\begin{equation}
f\in R_k \Leftrightarrow \frac{zf'(z)}{f(z)} \in P_k
\end{equation}
At the same time, we note that $V_k$ generalizes of convex functions.
\section{Main Results}
\begin{lem} Let $p(z)$ be an element of $P_k$, then
\begin{equation}
\bigg|p(z)-\frac{1+r^2}{1-r^2}\bigg|\leq \frac{kr}{1-r^2}
\end{equation}	
\end{lem}
\begin{proof} Let $f(z)$ be an element of $h(z) \in V_k$. Using (1.13), we can write
\begin{equation}
	p(z)=1+\frac{f''(z)}{f'(z)}, p(z)\in \pe_k
\end{equation}	
On the other hand M.S. Robertson \cite{Robertson} proved that if $f(z)\in V_k$, then
\begin{equation}	 
\bigg|z\frac{f''(z)}{f'(z)}-\frac{2r^2}{1-r^2}\bigg|\leq \frac{kr}{1-r^2}
\end{equation}
Therefore the relation can be written in the following form,
\begin{equation}
\bigg|(1+z\frac{f''(z)}{f'(z)})-\frac{1+r^2}{1-r^2}\bigg|\leq \frac{kr}{1-r^2}
\end{equation}
Using the definition of the class $V_k$, we obtain (2.1).
\end{proof}
\begin{thm}  Let $f(z)$ be an element of $R_k$, then
\begin{equation}
\frac{r}{(1-r)^{\frac{2-k}{2}}(1+r)^{\frac{2+k}{2}}} \leq |f(z)| \leq \frac{r}{(1-r)^{\frac{2+k}{2}}(1+r)^{\frac{2-k}{2}}} 
\end{equation}
\begin{equation}
\frac{1-kr+r^2}{(1-r)^{2-\frac{k}{2}}(1+r)^{2+\frac{k}{2}}} \leq |f'(z)| \leq \frac{1+kr+r^2}{(1-r)^{2+\frac{k}{2}}(1+r)^{2-\frac{k}{2}}} 
\end{equation}
\end{thm}
\begin{proof}  Using the definition of $R_k$, then we can write
\begin{equation}
\bigg|z\frac{f'(z)}{f(z)}-\frac{1+r^2}{1-r^2}\bigg|\leq \frac{kr}{1-r^2}
\end{equation} 
This inequality  can be written in the following form,
\begin{equation}
\frac{1-kr+r^2}{1-r^2} \leq Re z\frac{f'(z)}{f(z)} \leq \frac{1+kr+r^2}{1-r^2} 
\end{equation}
On the other hand, we have
\begin{equation}
Re z\frac{f'(z)}{f(z)}=r.\frac{\partial}{\partial r}log|f(z)|
\end{equation}
Thus we have
\begin{equation}
\frac{1-kr+r^2}{r(1-r^2)} \leq \frac{\partial}{\partial r}log|f(z)| \leq \frac{1+kr+r^2}{r(1-r^2)} 
\end{equation}
Integrating both sides (2.10), we get (2.5). The inequality (2.7) can be written in the form
\begin{equation}
\frac{1-kr+r^2}{1-r^2} \leq \bigg|z\frac{f'(z)}{f(z)}\bigg| \leq \frac{1+kr+r^2}{1-r^2} 
\end{equation}
In this step, if we use (2.5), we obtain (2.6).
\end{proof}
\begin{cor} For $k=2$ in (2.5), we obtain
$$\frac{r}{(1+r)^2} \leq |f(z)| \leq \frac{r}{(1-r)^2} $$
This is well known growth theorem for starlike functions \cite{Goodman 1984}.
\end{cor}
\begin{cor} For $k=2$ in (2.6), we obtain
$$\frac{1-r}{(1+r)^3} \leq |f'(z)| \leq \frac{1+r}{(1-r)^3} $$
This is well known distortion theorem for starlike functions \cite{Goodman 1984}.
\end{cor}
\begin{cor} The radius of starlikeness of $R_k$ is
\begin{equation}
R_{S^\ast}=\frac{k-\sqrt{k^2-4}}{2}, k\geq 2
\end{equation}
\end{cor}
\begin{proof} Since 
$$Re\bigg(z\frac{f'(z)}{f(z)}\bigg)>\frac{1-kr+r^2}{1-r^2} $$
\end{proof}
Hence for $R<R_{S^\ast}$ the left hand side of the preceding inequality is positive which implies (2.12).
We note that all results are sharp because of extremal function is 
$$f_\ast(z)=\frac{z(1-z)^{\frac{k}{2}-1}}{(1+z)^{\frac{k}{2}+1}}$$
Indeed,
$$z\frac{f_\ast'(z)}{f_\ast(z)}=\frac{1-kz+z^2}{1-z^2}=\left(\frac{k}{4}+\frac{1}{2}\right)\frac{1+z}{1-z}-\left(\frac{k}{4}-\frac{1}{2}\right)\frac{1-z}{1+z}$$
Thus, $f_\ast(z)\in R_k$ and $f_\ast(z)$ is extremal function.
\begin{lem} Let $p(z)=1+p_1z+p_2z^2+...$ be an element of $\pe_k$, then
$$|p_n|\leq k$$
\end{lem}
\begin{proof} Method I.
Since $p(z)\in \pe_k$, then we have
\begin{align*}
p(z)&=\left(\frac{k}{4}+\frac{1}{2}\right)p_{1}(z)-\left(\frac{k}{4}-\frac{1}{2}\right)p_{2}(z) \\
&=\left(\frac{k}{4}+\frac{1}{2}\right)(1+a_1z+a_2z^2+...)-\left(\frac{k}{4}-\frac{1}{2}\right)(1+b_1z+b_2z^2+...)
\end{align*}
Then we have
$$p_n=\left(\frac{k}{4}+\frac{1}{2}\right)a_n-\left(\frac{k}{4}-\frac{1}{2}\right)b_n$$
Thus
\begin{align*}
|p_n|&=\bigg|\left(\frac{k}{4}+\frac{1}{2}\right)a_n-\left(\frac{k}{4}-\frac{1}{2}\right)b_n\bigg| \\
&\leq\left(\frac{k}{4}+\frac{1}{2}\right)|a_n|+\left(\frac{k}{4}-\frac{1}{2}\right)|b_n|\\
&\leq\left(\frac{k}{4}+\frac{1}{2}\right)2+\left(\frac{k}{4}-\frac{1}{2}\right)2\\
\end{align*}
This shows that,
$$|p_n|\leq k$$

Method II.
Since $p(z)\in \pe_k$, then $p(z)$ can be written in the form
$$p(z)=\frac{1}{2\pi}\int^{2\pi}_{0}\frac{1+ze^{-it}}{1-ze^{-it}}d\mu(t)$$
and
$$\int^{2\pi}_{0} d\mu(t)=2\pi \quad\text{and}\quad \int^{2\pi}_{0} |d\mu(t)| \leq k\pi.$$
Then
\begin{align*}
p(z)=1+p_1z+p_2z^2+...&=\frac{1}{2\pi}\int^{2\pi}_{0}\frac{1+ze^{-it}}{1-ze^{-it}}d\mu(t)\\
&=\frac{1}{2\pi}\int^{2\pi}_{0}\frac{1+ze^{-it}-ze^{-it}+ze^{-it}}{1-ze^{-it}}d\mu(t)\\
&=\frac{1}{2\pi}\int^{2\pi}_{0}\bigg(1-\frac{2ze^{-it}}{1-ze^{-it}}\bigg)d\mu(t)\\
|p_n|&\leq \frac{1}{\pi}\int^{2\pi}_{0}|d\mu(t)|\leq k
\end{align*}
is obtained. 

We note that this lemma was proved first by K.I. Noor \cite{Noor} (Method II).
\end{proof}
\begin{thm} Let $f(z)$ be an element of $R_k$, then
\begin{equation}
|a_n|\leq \frac{1}{(n-1)!}\prod_{\nu=0}^{n-2}(k+\nu)
\end{equation}	
\end{thm}
\begin{proof} Since $f(z)\in R_k$, then we have
$$z\frac{f'(z)}{f(z)}=p(z)$$
where $p(z)\in \pe_k$. Thus
$$zf'(z)=f(z)p(z)$$
Comparing the coefficients in both sides of $zf'(z)=f(z)p(z)$, we obtain the recursion formula
$$a_n=\frac{1}{n-1}\sum_{\nu=1}^{n-1}p_{n-\nu}a_\nu, \quad n\geq 2$$
and therefore by Lemma 2.6,
$$|a_n|=\frac{k}{n-1}\sum_{\nu=1}^{n-1}|a_\nu|$$
Induction shows that
$$|a_n|\leq \frac{1}{(n-1)!}\prod_{\nu=0}^{n-2}(k+\nu).$$
\end{proof}
\begin{cor} For $k=2$, we obtain $|a_n|\leq n$. This inequality is well known coefficient inequality for starlike functions.
	
Indeed,
$$|a_n|\leq \frac{1}{(n-1)!}\prod_{\nu=0}^{n-2}(k+\nu)=\frac{k(k+1)(k+2)...(k+(n-2))}{(n-1)!}.$$
If we take $k=2$,
$$|a_n|\leq \frac{2.3.4...(n-2).(n-1).n}{(n-1)!}=n$$
\end{cor}
\begin{cor} Let $f(z)$ be an element of $V_k$, then
\begin{equation}
|a_n|\leq \frac{1}{n!}\prod_{\nu=0}^{n-2}(k+\nu)
\end{equation}	
\end{cor}	
\begin{proof} Using the theorem of Pinchuk
$$f(z)\in V_k\Leftrightarrow zf'(z)\in R_k$$
we get (2.14).
\end{proof}
\begin{cor} For $k=2$, we obtain $|a_n|\leq 1$. This inequality is well known coefficient inequality for convex functions.
\end{cor}
We note that all these inequalities are sharp because extremal function is,
$$f_\ast(z)=\frac{z(1-z)^{\frac{k}{2}-1}}{(1+z)^{\frac{k}{2}+1}}.$$

\footnotesize
\vspace{1cm}

\textsc{Ya\c{s}ar Polato\~{g}lu}\\
Department of Mathematics and Computer Sciences\\
\.{I}stanbul K\"{u}lt\"{u}r University, \.{I}stanbul, Turkey\\ 
e-mail: y.polatoglu@iku.edu.tr\\

\textsc{Yasemin Kahramaner}\\
Department of Mathematics,\\
\.{I}stanbul Ticaret University, \.{I}stanbul, Turkey \\
e-mail: ykahramaner@iticu. edu.tr\\

\textsc{Arzu Yemi\c{s}ci \c{S}en}\\
Department of Mathematics and Computer Sciences\\
\.{I}stanbul K\"{u}lt\"{u}r University, \.{I}stanbul, Turkey\\ 
e-mail: \\


\begin{thebibliography}{99}
\bibitem{Brannan 1969} D.A. Brannan, \textit{On functions bounded boundary rotation I}, Proc. Edinburg Math. Soc. 16 (1969), 339-347.
\bibitem{Goodman 1984} A.W. Goodman, \textit{Univalent functions Volume I and Volume II}, Mariner Pub. Co. Inc. Tampa Florida, 1984.
\bibitem{Loewner 1971} C.Loewner, \textit{Untersuchungen \"{u}ber die Verzerrung bei konformen Abbildungen des Einheitskreises $|z|<1$, die durch Funktionen mit nicht verschwindender Ableitung geliefert werden}, Ber. Verh. S\"{a}chs. Gess. Wiss. Leipzig, 69 (1917), 89-106.
\bibitem{Noor} K.I. Noor, \textit{ On generalization of close-to-convexity}, International Journal of Mathematics and Mathematical Sciences
Volume 6 (1983), Issue 2, 327-333.
\bibitem{Paatero 1931} V.Paatero, \textit{\"{U}ber die konforme Abbildung von Gebieten deren R\"{a}nder von beschr\"{a}nkter Drehung sind}, Ann. Acad. Sci. Fenn. Ser., A 33, (1931), 1-77.
\bibitem{Paatero 1933} V.Paatero, \textit{\"{U}ber Gebiete von beschr\"{a}nkter Randdrehung}, Ann. Acad. Sci. Fenn. Ser., A 37, (1933), 1-20.
\bibitem{Pinchuk 1971} B. Pinchuk, \textit{Functions with bounded boundary rotation}, Isr. J. Math., 10 (1971), 7-16.
\bibitem{Robertson} M.S. Robertson , \textit{Coefficients of functions with bounded boundary rotation}, Canad. J. Math., 21 (1969), 1477-1482
\end{thebibliography}
\end{document}